\documentclass{amsart} 
\usepackage{enumerate, amssymb, amsfonts, latexsym}
\usepackage{amscd, amsmath}
\usepackage{amsxtra}
\usepackage{graphicx}
\newtheorem{theorem}{Theorem} 
\newtheorem{axiom}{Axiom} 
 
\newtheorem{lemma}{Lemma} 
\newtheorem{proposition}{Proposition} 
\newtheorem{definition}{Definition}

\newtheorem{corollary}{Corollary}

\begin{document} 

\title{Supersymmetric Schr\"{o}dinger Operators with Applications to Morse Theory}
\author{Rohit Jain}

\begin{abstract}
In 1981 Edward Witten proved a remarkable result where he derived the classical Morse Inequalities using ideas from Supersymmetric (SUSY) Quantum Mechanics. In this regard, one has an example where a Physical Theory has something to say about the underlying Mathematical Structure. The objective of this survey paper is to revisit this classical result from the perspective of Schr\"{o}dinger Operators. The paper will be divided in four parts. The first part will revisit the classical theory of Morse and recall some of its fundamental results. In the second part, we consider the underlying physical motivations by considering Quantum Mechanics and 0-Dimensional SUSY. The third part will focus on Schr\"{o}dinger Operators and highlight some of their basic properties. Finally in the last section we will put everything together and present Witten's proof of the Morse Inequalities. Even here we must be completely honest and say that we only follow Witten in proving the Weak Morse Inequalities. The Strong Morse Inequalities are derived using related ideas from Supersymmetry, but mention is made of the techniques used by Witten to get at the Strong Morse Inequalities. 
\end{abstract}

\maketitle
\tableofcontents

\section{Disclaimer}
In this paper we consider Witten's remarkable result that using Supersymmetric Quantum Mechanics  one can derive the classical Morse Inequalities. It is important to note that this is only one of the many ideas present in Witten's paper, "Supersymmetry and Morse Theory". In fact using arguments similar to the ones he presents when proving the Morse Inequalities, Witten was also able to derive the Lefschetz Fixed Point Theorem as well as provide lower bounds on the Betti numbers. The eventual goal Witten had in mind was to extend these results to the infinite dimensional setting where one considers Supersymmetric Quantum Field Theory. Here we only focus on Witten's derivation of the Morse Inequalities. In deriving the Morse Inequalities, Witten considers both the case of a degenerate Morse Function and a non-degenerate Morse Function. We will restrict our attention in this paper to the situation where we have a non-degenerate Morse Function. Given our limited focus we hope to leave these other results to future ruminations. Moreover there are no claims of originality in this paper. 

\section{Introduction to Morse Theory}
Let $f$ be a real-valued smooth function on a smooth manifold $M$. The basic insight of Morse Theory is that such a function $f$ can provide us with information about the underlying topological structure of the Manifold $M$.  For each $x \in M$, $f(x)$ is either a regular value(i.e. the derivative is surjective at the point $x$), or $f(x)$ is a critical value. The optimal local behavior of the function $f$ around a regular point can be completely understood up to diffeomorphism by the inverse function theorem. The question remains how to understand the set: 
$$C(f) = \{x \in M \; : \; df(x) = 0 \}$$  
The content of Morse Theory is to shed light on the relationship between this set and the global topology of $M$.  We make the following assumption before beginning our analysis. $M$ will be a finite dimensional compact oriented Riemannian Manifold. There is an interesting point to make that Morse Theory is not just restricted to finite dimensional manifolds. In fact under various extensions of critical point theory, such as the Palais-Smale condition, one can extend Morse Theory to the infinite dimensional setting. We do not pursue these ideas here. Instead we start by non-trivally restricting the set of smooth functions we will consider: 
\begin{definition} $f \in C^{2}(M,\mathbb{R})$ is called a Morse Function if for every $x_{0} \in C(f)$ the Hessian $d^{2}f(x_{0})$ is nondegenerate. 
\end{definition}
Recall that the Hessain $d^{2}f(x)$ of a differentiable function $f:M \to \mathbb{R}$ on a Riemannian Manifold M is defined to be: $$d^{2}f(x) = \nabla df$$ where $\nabla$ is the Levi-Civita Connection and $df = \sum_{i} \frac{\partial f}{\partial x^{i}}dx^{i}$ is the exterior derivative acting on zero forms. The non-degeneracy condition can be translated to mean that there exists a continuous linear operator 
$$L: T_{x_{0}}M \to T^{*}_{x_{0}}M$$  
defined by 
$$A_{u}(v) = d^{2}f(x_{0})(u,v) \; \; \forall u,v \in T_{x_{0}}M$$
One can also define $$V^{-} \subseteq T_{x_{0}}M$$ to be the subspace spanned by the eigenvectors of $d^{2}f(x_{0})$ (viewed as a bounded, symmetric, bilinear form) with negative eigenvalues. This allows us to consider: $$\mu(x_{0}) = dim(V^{-})$$ the $\textit{Morse Index}$ of $x_{0} \in C(f)$. We also let: 
$$M_{p}(f) = \textnormal{Number of critical points with Morse Index p}$$
With obvious modifications these definitions extend to the infinite dimensional situation. An additional comment worth making is that one can check that the non-degeneracy condition and the Morse Index are independent of the choice of coordinates. This can be verified directly, or by formulating a coordinate invariant definition for the Hessian. 
\\
\\
Because of the non-degenracy character of the Hessian, one can give a local characterization of Morse Functions around any one of the critical points:
\begin{lemma}(Morse Lemma) Let $f:M \to \mathbb{R}$ be of class $C^{2}$ and let p be a non-degenerate critical point of f. The there exists a system of coordinates $(y_{1}, _{\spdddot}, y_{n})$ near p such that in these coordinates: $$ f(y) = f(p) - y_{1}^{2} - _{\spdddot} - y_{\mu}^{2} + y_{\mu + 1}^{2} + _{\spdddot} + y_{n}^{2}$$ Where $\mu$ is the Morse Index of $f$ at p. Hence since M is compact we only have finitely many non-degenerate critical points. 
\end{lemma}
In order to understand more fully why Morse Functions have anything useful to say about the underlying topological structure of the manifold M, we will have to start by recalling some facts from the topological study of Manifolds and state some fundamental results in Morse Theory. The motivation for this exposition is to convey the following two general remarks: 
\\
\\
1.) If one considers the set $$M^{c} = \{x \in M \; : \; f(x) \leq c\}$$ the topology of $M^{c}$ does not change as we vary $c$ unless upon varying we pass through a critical point.   
\\
\\
2.) When passing through a critical point the Morse Index of the critical point completely determines how the topology of $M$ changes. 
\\
\\
The following will be a brief recollection of basic constructions in Algebraic Topology. We follow the discussion in \cite{AM}. 

\begin{definition} Let $(G_{i})_{i}$ be a sequence of Abelian Groups and let $(\psi_{i})_{i}$ be a sequence of homomorphisms: $$_{\spdddot} \to G_{i} \stackrel{\psi_{i}}{\to} G_{i+1} \stackrel{\psi_{i+1}}{\to}  G_{i+2} \to \; _{\spdddot}$$
The sequence is exact if $ \forall i \; im(\psi_{i}) = ker (\psi_{i+1})$ 
\end{definition}
For example if we let $G_{1}, G_{2}$ be Abelian Groups, and consider the following sequence: 
$$0  \to G_{1} \stackrel{\psi}{\to} G_{2} \to 0$$
Then we have exactness if and only if $\psi$ is an isomorphism. 
\\
\\
We will let $(X,A)$ denote a Topological Space $X$ with $A \subseteq X$. Also $(X,A) \subseteq (Y,B)$ means $X \subseteq Y$ and $A \subseteq B$. Let $f: (X,A) \to (Y,B)$ denote a continuous function from $X$ to $Y$ with $f(A) \subseteq f(B)$. We now Axiomatize the notion of Homology Groups.
\begin{definition} (Homology Groups)
\\
(a) $\forall q \in \mathbb{Z}$ and for every pair $(X,A)$ we associate a group $H_{q}(X,A)$
\\
(b) To every map $f:(X,A) \to (Y,B)$ we associate a homomorphism $f_{*}: H_{q}(X,A) \to H_{q}(Y,B)$
\\
(c) $\forall q \in \mathbb{Z}$ and every pair $(X,A)$ we associate a homomorphism $\partial : H_{q}(X,A) \to H_{q-1}(A)$
\end{definition}
Let us now present the Axioms we need to study Homology Groups:
\begin{axiom} $f = Id_{X} \implies f_{*} = Id_{H_{q}(X,A)}$
\end{axiom}
\begin{axiom} Let $f:(X,A) \to (Y,B)$ and $g: (Y,B) \to (Z,C)$ then $(g \circ f)_{*} = g_{*} \circ f_{*}$
\end{axiom}
\begin{axiom} $f:(X,A) \to (Y,B) \implies \partial \circ f_{*} = (f|_{A})_{*} \circ \partial$
\end{axiom} 
\begin{axiom} Let $i: A \to X$ and $j: (X, \varnothing) \to (X,A)$ be the inclusion maps, then the sequence: $$_{\spdddot}  \stackrel{\partial}{\to}  H_{q}(A) \stackrel{i_{*}}{\to} H_{q}(X) \stackrel{j_{*}}{\to}  H_{q}(X,A) \stackrel{\partial}{\to}  H_{q-1}(A) \to \; _{\spdddot}$$ is exact.
\end{axiom}
\begin{axiom} If $f,g:(X,A) \to (Y,B)$ are homotopic maps (i.e. $\exists h: [0,1] \times (X,A) \to (Y,B)$ such that $h(0, \centerdot) = f$ and $h(1, \centerdot) = g$.) then $f_{*} = g_{*}$. 
\end{axiom}
\begin{axiom} (Excision) If $U$ is an open set of $X$ with $\bar{U} \subseteq int(A)$ and $i: (X \setminus U, A \setminus U) \to (X,A)$ is the inclusion map, then $i_{*}$ is an isomorphism
\end{axiom}
\begin{axiom} Let $X = \{p\}$ then: 
$$ H_{q}(\{p\}) =  
\begin{cases}
\mathbb{Z}, & \text{if }q = 0 \\
0, & \text{if }q\neq 0
\end{cases}
$$
\end{axiom}
Using these axioms one can in principle derive many of the usual properties about Homology groups. In particular one can prove:
\begin{theorem} (Mayer-Vietoris Sequence) Let $X = X_{1} \cup X_{2}$ with $X_{1}$,$X_{2}$ open sets. Let $A_{1} \subseteq X_{1}$ and $A_{2} \subseteq X_{2}$ also open. Then if $X_{1} \cap X_{2} \neq \varnothing$ letting $A = A_{1} \cup A_{2}$ and defining $i: (X_{1},A_{1}) \to (X,A)$, $j: (X_{2}, A_{2}) \to (X,A)$, $l: (X_{1} \cap X_{2},A_{1} \cap A_{2}) \to (X_{1},A_{1})$ and  $k: (X_{1} \cap X_{2},A_{1} \cap A_{2}) \to (X_{2},A_{2})$ the natural inclusions, there exists an exact sequence: 
$$_{\spdddot} \to  H_{q}(X_{1},A_{1}) \oplus H_{q}(X_{2},A_{2})  \stackrel{\psi}{\to} H_{q}(X,A) \to  H_{q-1}(X_{1} \cap X_{2},A_{1} \cap A_{2})$$ 
$$\stackrel{\varphi}{\to}  H_{q-1}(X_{1},A_{1}) \oplus H_{q-1}(X_{2},A_{2})  \to \; _{\spdddot}$$
where $\psi = i_{*} - j_{*}$ and $\varphi = (k_{*}, l_{*})$. 
\end{theorem}

\begin{lemma} (Deformation Retraction) If $A \subseteq X$ is a deformation retract of X (i.e. $\exists h: [0,1] \times (X)$ such that $h(0, \centerdot) = Id_{X}$, $h(1, \centerdot) \subseteq A$ and $h(t, \centerdot)_{A} = Id_{A} \;  \forall t$) then $H_{q}(A) = 0 \; \forall q \in \mathbb{Z}$.     
\end{lemma}
Using these two results one can prove that:
$$ H_{q}(S^{n}) =  
\begin{cases}
\mathbb{Z}, & \text{if }q = 0,n \\
0, & \text{otherwise}
\end{cases}
$$
Let us now construct an explicit Homology Theory. 
\begin{definition} For every non-negative integer r, we define the simplex, $s_{r} \subseteq \mathbb{R}^{n+1}$: 
$$s_{r} = \{ \; t_{0}\mathbf{e}_{0} + _{\spdddot} + t_{r}\mathbf{e}_{r} \; : \; \sum_{i = 0}^{r}t_{i} = 1\: \}$$
where $\mathbf{e}_{i} = (0,0, _{\spdddot},0,1,0,_{\spdddot},0)$ in the $i^{th}$ slot. A singular r-simplex of a topological space $X$ is a continuous map $\sigma_{r}: s_{r} \to X$. 
\end{definition}

Let $C_{r}$ denote the Abelian Group generated by linear combinations (with integer coefficients) of singular r-simplixes of $X$. These will be called the singular r dimensional chains of $X$.  
\\
\\
Given $\sigma_{r}: s_{r} \to X$, we define the $j^{th}$ face of $\sigma_{r}$ by restricting the action of $\sigma_{r}$ to the set: 
$$s_{r-1}^{j} = \{ \; t_{0}\mathbf{e}_{0} + _{\spdddot} + t_{r}\mathbf{e}_{r} \in s_{r} \; : \; t_{j} = 0 \; \}$$
We denote this restriction by $\sigma_{r-1}^{j}$. We can also define the boundary operator $\partial_{r}$ as the linear extension of the map: $$\partial_{r} \sigma_{r} = \sum_{j = 0}^{r} (-1)^{j}\sigma_{r-1}^{j}$$ This boundary map satisfies the additional property: $$\partial_{r-1} \circ \partial_{r} = 0$$ Hence using the boundary operator, the singular chains form a Homological Complex: 
$$_{\spdddot} \to  C_{r+1}  \stackrel{\partial_{r+1}}{\to} C_{r} \stackrel{\partial_{r}}{\to}  C_{r-1} \to _{\spdddot} \to C_{1} \stackrel{\partial_{1}}{\to} C_{0} \stackrel{\partial_{0}}{\to} 0$$ 
In this case the maps $\partial_{r}$ are homomorphisms and the composition of two consecutive maps vanish.  We also define:
$$Z_{r}(X) = ker(\partial_{r}) \subseteq C_{r} \; \; \; \textnormal{[Cycles]}$$
$$B_{r}(X) = im(\partial_{r+1}) \subseteq Z_{r} \; \; \; \textnormal{[Boundaries]}$$
The $r^{th}$ singular homology group of X is defined to be the quotient: 
$$H_{r}(X) = Z_{r}(X) \setminus B_{r}(X)$$
This entire sequence is usually denoted by:
$$H_{*}(X)$$
Finally we conclude our presentation of Homology Theory by introducing Singular Relative Homology Groups. Let $X$ be a Topological Space and $A\subseteq X$. The inclusion map $i: A \to X$ induces a homomorphism $i_{*}: C_{q}(A) \to C_{q}(X)$ which is trivially injective. Consider the following quotient group: 
$$C_{q}(X,A) = C_{q}(X) \setminus C_{q}(A)$$
Since $i_{*}$ commutes with $\partial$ we have an induced homomorphism which we still denote by $\partial$. 
$$\partial: C_{q}(X,A) \to C_{q-1}(X,A)$$ In an analogous manner we can define: 
$$Z_{q}(X,A) = ker(\partial_{q}) \; \; \; \textnormal{[Relative cycles]}$$
$$B_{q}(X,A) = im(\partial_{q+1})  \; \; \; \textnormal{[Relative boundaries]}$$ 
We also define the relative Homology Group by the obvious construction:
$$H_{q}(X,A) = Z_{q}(X,A) \setminus B_{q}(X,A)$$
We now turn to reviewing some classical quantities in terms of the Homology Theory we have developed:
\begin{definition} The rank of an Abelian Group G is the maximal number k, for which $\sum_{i=1}^{k} n_{i}g_{i} =0$ with $(n_{i})_{i} \subseteq \mathbb{Z}$ and $(g_{i})_{i} \subseteq G$ implies $g_{i} = 0$  for all i. 
\end{definition}
\begin{definition} Given a pair of spaces $(B,A)$ and $q \in \mathbb{Z}$ we let: 
$$\beta_{q}(B,A) = rank(H_{q}(B,A))$$
$$\chi(B,A) = \sum (-1)^{q}\beta_{q}(B,A)$$
\end{definition}
Here we are assuming the ranks are finite and in the definition of $\chi(B,A)$ that $\beta_{q}$ is nonzero except for a finite number of q. $\beta_{q}(B,A)$ is called the q-th Betti Number of $(B,A)$ and $\chi(B,A) $ is known as the Euler-Poincar\'{e} characteristic of (B,A). 
\\
\\
We have the following results about the Betti Numbers and Euler-Poincar\'{e} Characteristic:
\begin{lemma} If $A \subseteq B \subseteq C$ and $q \in \mathbb{Z}$ then,
$$\beta_{q}(C,A) \leq \beta_{q}(C,B) + \beta_{q}(B,A)$$
$$\chi(C,A) = \chi(C,B)  + \chi(B,A)$$
Let $\mathcal{B}(B,A) = \beta_{q}(B,A) - \beta_{q-1}(B,A) + _{\spdddot} \pm  \beta_{0}(B,A)$ then,
$$\mathcal{B}_{q}(C,A) \leq \mathcal{B}_{q}(C,B) + \mathcal{B}_{q}(B,A)$$ 
Suppose also that we have the inclusion $X_{0} \subseteq X_{1} \subseteq _{\spdddot} \subseteq X_{n}$ then by induction,
$$\beta_{q}(X_{n},X_{0}) \leq \sum_{i=1}^{n} \beta_{q}(X_{i},X_{i-1})$$
$$ \mathcal{B}(X_{n},X_{0}) \leq  \sum_{i=1}^{n} \mathcal{B}_{q}(X_{i},X_{i-1})$$  
$$ \chi(X_{n},X_{0}) =  \sum_{i=1}^{n} \chi(X_{i},X_{i-1})$$  
\end{lemma}

Having reviewed the Algebraic Topological perspective about Manifolds, let us now return to the classical theorems of Morse Theory. 
\\
\\
The first classical result about Morse Theory makes precise the first observation we made at the beginning of the section: 
\begin{proposition} Let M be a compact, finite dimensional manifold, and let $f: M \to \mathbb{R}$ be a function of class $C^{2}$. Let $a,b \in \mathbb{R}$ with $a < b$ and assume that $\{a\leq f \leq b \}$ contains no critical points of f. Then $M^{a}$ is a deformation retract of $M^{b}$.
\end{proposition}

The idea of the proof is to use a gradient flow to construct an explicit homotopy (deformation retraction) between $M^{a}$ and $M^{b}$. The $C^{2}$ assumption is what allows the gradient flow to have a local solution. 
\\
\\
Now if upon varying the level sets we do encounter a critical point, the fundamental result in this direction tells us that the resulting level set after deforming has the homotopy type of the original level set with a $\lambda$-dimensional cell attached. Furthermore the dimension of the cell is precisely the Morse Index of the critical point. 
\begin{proposition} Suppose p is a nondegenerate critical point of f with Morse Index $\mu$ and suppose there exists $\epsilon >0$ such that f has no critical points between $(c-\epsilon, c+\epsilon)$ (except p) where $c = f(p)$. Then $M^{c+\epsilon}$ has the homotopy type of $M^{c-\epsilon}$ with a $\mu$-dimensional cell attached. 
\end{proposition}
One can construct this handle using the local description of the Morse function, where the attaching map is between the boundary of the $\mu$-dimensional cell and $\partial M^{c-\epsilon}$.  One either uses Homotopy type arguments using technical results of Whitehead to obtain a cell complex, or a variation of the Mayer-Viteoris argument. As a corollary we get:
\begin{corollary} Let p be a nondegenerate critical point of f with Morse Index $\mu$. Let $c = f(p)$ and assume it is the only critical point with energy level c. Then for sufficiently small $\epsilon$:
$$ H_{q}(M^{c+\epsilon},M^{c-\epsilon}) =  
\begin{cases}
\mathbb{Z}, & \text{for }q = \mu \\
0, & \text{for} \; q \neq \mu
\end{cases}
$$
\end{corollary}
Note that this result can be strengthened to the case where the function $f$ has $z_{1}, _{\spdddot},z_{m}$ critical points at the level c and all of them are non-degenerate and of Morse index $\mu_{1}, _{\spdddot}, \mu_{m}$. Then $M^{c+\epsilon}$ has the homotopy type of $M^{c-\epsilon}$ with m cells attached, each having the dimension of  the respective Morse Index of the critical point under consideration. Moreover for $\epsilon$ sufficiently small, $H_{q}(M^{c+\epsilon},M^{c-\epsilon}) \simeq \mathbb{Z}^{i_{q}}$ where $i_{q} = \textnormal{number of critical points with Morse Index q}$. 	
\\

We now arrive at the central mathematical statement of this paper:
\begin{theorem} (Morse Inequalities) Let M be a compact, finite dimensional manifold, with a Morse Function $f:M \to \mathbb{R}$. Then for any non-negative integer q we have the following relations:
$$M_{p}(f) \geq \beta_{p}(M) \; \; \; \textnormal{[Weak Morse Inequalities]}$$
$$  \sum_{i \geq 0} (-1)^{q-i} M_{i} \geq \sum_{i \geq 0} (-1)^{q-i} \beta_{i} \; \; \; \textnormal{[Strong Morse Inequalities]}$$
$$  \sum_{i \geq 0} (-1)^{i} M_{i} = \chi(M) \;  \; \textnormal{[Morse Index Theorem]}$$
Where $M_{p}(f)$ is the number of critical points of f with Morse Index equal to p. 
\end{theorem}
\begin{proof}
Let $c_{1} < c_{2} < _{\spdddot} < c_{k}$ denote the critical levels of f, which are only finitely many by the compactness of our Manifold. Choose real numbers $\alpha_{0}, \alpha_{1}, _{\spdddot}, \alpha_{k}$ such that $\alpha_{i} < c_{i} < \alpha_{i+1}$ $\forall \; 0 \leq i \leq k-1$. In particular we have $M^{\alpha_{0}} = \varnothing$ and $M^{\alpha_{k}} = M$. By  Proposition 1, we have that for any integers $i,q$ and any small $\epsilon >0$ $H_{q}(M^{c_{i}+\epsilon},M^{c_{i}-\epsilon}) \simeq H_{q}(M^{a_{i}},M^{a_{i-1}})$. Hence applying the above remark following the corollary and Lemma 3, we conclude the proof.   
\end{proof}
Before leaving the subject let us just mention an alternate way to describe the classical Morse Inequalities. 
\begin{definition} The Poincar\'{e} Polynomial of M is defined to be 
$$ P_{t}(M) = \sum_{k = 0}^{m} \beta_{k}t^{k}$$
and the Morse Polynomial of f is defined to be
$$N_{t}(f) = \sum_{k = 0}^{m} M_{k}t^{k}$$
\end{definition}
\begin{theorem}(Polynomial Morse Inequalities) For any Morse function f on a smooth manifold M
$$N_{t}(f) = P_{t}(M) + (1+t)Q(t)$$
Where $Q(t)$ is a polynomial with non-negative integer coefficients.
\end{theorem}
Instead of proving this result directly we instead prove that this statement is equivalent to the Morse Inequalities we proved above.
\begin{theorem} Morse inequalities $\iff$ Polynomial Morse Inequalities
\end{theorem}
\begin{proof}
$$ (\implies)$$
Notice that $N_{-1}(f) = P_{-1}(M)$. Hence $t = -1$ is a root of the polynomial $N_{t}(f) - P_{t}(M)$. By the Euclidean Division Algorithm this implies that $P_{t}(M) + (1+t)Q(t) = N_{t}(f)$ for some polynomial $Q(t) = \sum_{k = 0}^{n-1}q_{k}t^{k}$. It is clear that $q_{k} \in \mathbb{Z}$ $\forall k = 0, _{\spdddot}, n-1$ since both $N_{t}(f)$ and $P_{t}(M)$ have integer coefficients. It remains to show that the coefficients are all non-negative. 
\\
\\
$P_{t}(M) + (1+t)Q(t) = N_{t}(f)$ implies that $M_{0} = \beta_{0} + q_{0}$. Next we also have $M_{1} = \beta_{1} - \beta_{0} + q_{1}$. Both these results follow from matching coefficients. They imply together $M_{1} = \beta_{1} + q_{1} + M_{0} - \beta_{0}$ or $M_{1} - M_{0} = \beta_{1} - \beta_{0} + q_{1}$. Continuing to match coefficients we get
$$M_{k} - M_{k-1} + _{\spdddot} + (-1)^{k}M_{0} = \beta_{k} - \beta_{k-1} + _{\spdddot} + (-1)^{k}\beta_{0} + q_{k} \; \; \; \forall k \in 0,_{\spddot},n-1$$
By the Strong Morse Inequalities we conclude that $q_{k} \geq 0$ $\forall k = 0, _{\spdddot}, n-1$

$$ (\Longleftarrow) $$

We assume that $P_{t}(M) + (1+t)Q(t) = N_{t}(f)$ for some polynomial $Q(t) = \sum_{k = 0}^{n-1}q_{k}t^{k}$ with non-negative coefficients. This inequality implies:
$$M_{k} - M_{k-1} + _{\spdddot} + (-1)^{k}M_{0} = \beta_{k} - \beta_{k-1} + _{\spdddot} + (-1)^{k}\beta_{0} + q_{k} \; \; \; \forall k k = 0, _{\spdddot}, n-1$$
Since $q_{k} \geq 0$ $\forall k k = 0, _{\spdddot}, n-1$ we get the first of the strong Morse Inequalities. To recover the second equality simply substitute $t = -1$.
\end{proof}

Now that we have shown both the classical Morse Inequalities and their equivalence to the Polynomial Morse Inequalities, the goal for the rest of this section will be to introduce the much needed Geometric Analysis framework to help in understanding the connection Witten makes with Supersymmetric Quantum Mechanics. For this undertaking it will be wise to recall the theory of differential forms.
\\
\\
Given a Homological group one can define a dual notion of a Cohomological group. Above we presented an axiomatic construction of the Homological complex using singular chains. By applying duality arguments one can construct a cochain complex and define analogously cocycles and coboundaries. The heuristic to keep in mind is that the arrows go in the reverse direction. More concretely, for our purposes, we are interested in the de Rham Cohomology group which one can construct using the space of smooth p-forms, $\varOmega^{p}(M)$, and the exterior derivative $d$. In fact we have the following cochain complex:    
$$0\to  \varOmega^{0}(M)  \stackrel{d}{\to} \varOmega^{1} \stackrel{d}{\to}  \varOmega^{2} \to _{\spdddot} \to \varOmega^{n} \stackrel{d}{\to} 0$$
Using the properties of the exterior derivative it can be checked that we have a similar vanishing condition $d \circ d =0$. Hence, we can define a quotient space: 
$$H^{p} = ker(d: \varOmega^{p} \to \varOmega^{p+1}) \setminus im(d: \varOmega^{p-1} \to \varOmega^{p})$$
This is called the pth de Rham cohomology group of M. 
\\
\\
Using the Riemannian Structure of M one can construct a scalar product on each $T_{x}^{*}(M)$.  We define the Hodge Star operator:
$$\star: \varOmega^{p}(M) \to \varOmega^{n-p}(M)$$
Where: 
$$Vol(M) = \int_{M} \star(1)$$
Using this construction we define the $L^{2}$ inner-product of p-forms, $\alpha$ and $\beta$: 
$$<\alpha, \beta> = \int_{M} \alpha \wedge \star \beta$$
The completion of $\varOmega^{p}$ under this $L^{2}$ inner product is a Hilbert Space which we will continue to denote as $\varOmega^{p}$. Using this inner-product one further defines the formal adjoint to the $d$ operator on the space $\bigoplus_{p =o}^{n} \varOmega^{p}(M)$ (a Geometric realization of the Riesz Representation theorem): 
$$\forall \alpha \in \varOmega^{p-1}(M), \; \; \beta \in \varOmega^{p}(M)$$
$$<d\alpha, \beta> = <\alpha, d^{*} \beta>$$
In particular $$d^{*}: \varOmega^{p}(M) \to \varOmega^{p-1}(M)$$ One can check that 
$$d^{*} = (-1)^{d(p+1)+1}\star d \star$$
The combination of $d$ and its adjoint $d^{*}$ allows us to construct arguably the most important operator for our analysis: 
\begin{definition} The Laplace-Beltrami Operator on $\varOmega^{p}(M)$ is 
$$\Delta = dd^{*} + d^{*}d: \varOmega^{p}(M) \to \varOmega^{p}(M)$$ Furthermore $\omega \in \varOmega^{p}(M)$ is called harmonic if: 
$$ \Delta \omega = 0$$
\end{definition}

Using the structure of the inner products it can be verified that the Laplacian is a positive self-adjoint operator on the space of forms. Having defined the fundamental operators on our Riemannian Manifold, we now come to the basic result about Harmonic forms and Cohomology: There exists  a canonical representative in each cohomology class of $H^{p}$. More specifically:
\begin{theorem} (Hodge) Let M be a compact Riemannian Manifold. Then every cohomology class in $H^{p}$ $(0\leq p \leq n)$ contains precisely one harmonic form. 
\end{theorem}
This is a remarkable result that one cannot hope to do justice to in such a short paper. Instead as we will see, Witten's result can be seen as a Physical realization of this basic representation. For the time being we content ourself by using this result to obtain the dimension of the de Rham Cohomology groups
\begin{corollary} Let M be a compact, oriented, differentiable Manifold. Then all cohomology groups $H^{p}$ $(0 \leq p \leq n)$ are finite dimensional
\end{corollary}
\begin{proof}
Suppose by contradiction that $H^{p}$ is infinite dimensional. Then we consider a sequence of harmonic forms $(\eta_{n})_{n\in \mathbb{N}} \subseteq H^{p}(M)$. Specifically: 
$$<\eta_{n}, \eta_{m}> = \delta_{nm} \; \; \forall n,m \in \mathbb{N}$$
By harmonicity we also know $d^{*}\eta = d\eta = 0$. By Rellich-Kondrakov Compactness Theorem we know that upto a subseqence our original sequence converges in $L^{2}$ to some $\eta$. But according to our choice of an orthonormal sequence 
$$ \|\eta_{n} - \eta_{m}\|_{L^{2}} = <\eta_{n}, \eta_{m}>^{2} \geq 1$$ Hence it is not a Cauchy sequence and cannot converge. A contradiction.  
\end{proof}

The dimension of this vector space has a most familiar sounding name:
\begin{definition} The p-th Betti Number of M is defined to be $\beta_{p}(M) = dim(H^{p})$
\end{definition}
The choice of name for this dimension should alert the reader to the fact that a p-th Betti Number has already been defined! But it is a remarkable fact of Homology/Cohomology theory that the two notions coincide! In fact using some intense results from Homological Algebra one can prove the "Universal Coefficients theorem for Cohomology" which gives a noncanonical isomorphism between $H_{i}(M)$ and $H^{i}(M)$. Hence the dimensions are the same and the definitions of the Betti numbers coincide. One final remark to make is that we will consider the p-th Betti number not as the dimension of $H^{p}(M)$ but as the dimension of the kernel of the Laplacian restricted to p-forms. This is one of the many things Hodge's Theorem tells us. 
\\
\\
For our upcoming application to Supersymmetric Quantum Mechanics we fix the following notation:
$$Q = d + d^{*}$$
$$\mathcal{H} = \bigoplus_{p = 0}^{n} \varOmega^{p}$$
$$P_{\varOmega^{p}} = (-1)^{p}$$
Note: 
$$Q^{2} = \Delta$$
$$P^{2} = 1$$
$$\{Q,P\} = 0$$
\\
In the coming attractions we will see that $Q$ will represent a Supersymmetry generator, $\mathcal{H}$ is going to be the Hilbert Space representing our State Space, and $\{\cdot, \cdot \}$ is the anti-commutator defined to be 
$$\{A, B\} = AB + BA$$ 
All of this will hopefully be better explained in the next section. 
 
\section{Quantum Mechanics and 0-Dimensional SUSY}
Quantum Mechanics is in many ways the most accurate Physical Theory we have to date. Two physical consequences of the theory are that, matter behaves in a random way, and matter exhibits wave like properties. More specifically, the behavior of individual particles is intrinsically random, and this randomness is propagated according to the laws of wave mechanics. 
\\
\\
In Quantum Mechanics the state of a particle is described by a complex-valued function of both space and time: $\psi(x,t)$ where $x \in \mathbb{R}^{n}$ and $t \in \mathbb{R}$. Given our two physical assumptions our wave function satisfies the following properties:
\\
a) $|\psi(x,t)|^{2}$ represents the probability distribution of the particle's position. 
\\
b) $\psi(x,t)$ solves a Wave Equation
\\
\\
Because the modulus square of our particle's wave function represents a probability distribution, it seems natural to require that the space of all possible states of our particle should have finite $L^{2}$ norm. Specifically we define the State Space of our particle to be:
$$L^{2}(\mathbb{R}^{n})$$ One immediate property we get is that the State Space has the nice property of being a Hilbert Space. Turning our attention to the second requirement, there is a particular wave equation that has withstood the test of time, as both a Mathematical object and its cordial relationship with the Experimental Physicist (of course we mean to say that in the domain of non-relativistic Quantum Theory it has been experimentally verified). The wave mechanics for our Quantum Mechanical particle is determined by the Schr\"{o}dinger Equation:
$$i\hbar \frac{\partial}{\partial t}\psi = -\frac{\hbar^{2}}{2m} \Delta \psi + V\psi$$
Where $\Delta$ is the Laplacian in Euclidean Space, V is the Potential Energy, $\hbar$ is Planck's constant, and m represents the mass of the particle. 
\\
\\
So if one wishes to understand the dynamics of a Quantum Mechanical particle, one must study the theory of operators on a Hilbert Space. From the analytical perspective the basic mathematical problem is to understand the Cauchy Problem:
$$i\hbar \frac{\partial}{\partial t}\psi = -\frac{\hbar^{2}}{2m} \Delta \psi + V\psi$$
$$ \psi(x,0) = \psi_{0} $$
This is an initial value problem for a Partial Differential Equation defined on a Hilbert Space. What one first looks at is the existence problem for solutions given that our initial data lies in a suitable function space. What is important for this question is a subtle notion of Self-Adjointness. The condition basically translates to the well known symmetric property for operators, formulated in terms of the inner product, along with some technical conditions regarding the operators domain. A fundamental result in the Mathematical theory of Quantum Mechanics gives us the conditions under which we expect a solution to the Cauchy Problem:
\begin{theorem} (Existence of Dynamics) We say that dynamics exists if the Cauchy Problem has a unique solution which conserves probability. Dynamics exists if and only if $H =  -\frac{\hbar^{2}}{2m} \Delta + V$ is a self-adjoint operator on our State Space.
\end{theorem}
Trying to understand when dynamics exists is an interesting problem in itself, but from the Physicist's point of view, one is also very interested in measuring quantities about the physical system. For example, determining the position of the particle at a particular time is one important quantity we made mention of before. Two other quantities of specific interest are the particle's energy and the particle's momentum. These types of quantities are called the observables of the physical system. 
\begin{definition} An observable is a self-adjoint operator on the state space $L^{2}(\mathbb{R}^{n})$. 
\end{definition}
Regarding the energy of the system, we have actually already made mention of this observable. In fact it is nothing else but the right hand side of Schr\"{o}dinger's Equation. This operator is called the Hamiltonian of our system. It is usually denoted by $H$. 
\\
\\
Fundamental to the analysis of observables is the commutator bracket. The commutator for two bounded operators on a Hilbert Space is defined to be:
$$[A,B] = AB -BA$$
With some technical assumptions one can extend this notion to operators that are unbounded as well. What one needs is a notion of Spectral Measure (Actually these guys will creep up again when we discuss the Physical notion of Supersymmetry in the context of Quantum Mechanics). 

In one particular representation (called the Heisenberg Representation) of Quantum Mechanics the state of the particle is fixed for all times. Instead what one considers is the time evolution of the observable. In this formulation if A is the observable of the system, A satisfies the following differential equation: 
$$\frac{d}{dt}A(t) = \frac{i}{\hbar}[H,A(t)]$$  

The Schr\"{o}dinger picture (Wave Mechanics) and the Heisenberg Representation are equivalent ways of understanding Quantum Mechanics. In fact one can show that they are equivalent upto a unitary transformation. There is actually a third equivalent way to view Quantum Mechanics and it is due to Richard Feynman. The Feynman Path Integral is an object that has had a profound influence on results in Mathematics. It has served as a heuristic for Physicists and as a means to uncover deep results in Topology and Geometry. Instead of treading in these waters, we offer a more modest perspective. The objective of what is to follow will be to give an interpretation of the Feynman Path Integral as a convenient representation for the integral kernel of the Schr\"{o}dinger evolution operator, $e^{\frac{-itH}{\hbar}}$. This operator when applied to the initial data of the Cauchy problem, $\psi_{0}$ is the unique solution to the Schr\"{o}dinger equation. It can be conveniently written as
$$\psi(x,t) = e^{\frac{-itH}{\hbar}} \psi_{0}$$
Let $K_{n} = e^{\frac{it\hbar}{2mn}\Delta}e^{-\frac{-itV}{\hbar n}}$. We have by the Trotter Product Formula 
$$e^{\frac{-itH}{\hbar}} = \lim_{n \to \infty} (e^{\frac{it\hbar}{2mn}\Delta}e^{-\frac{-itV}{\hbar n}})^{n}$$      
This allows us to define an integral kernel for the operator $K_{n}$. Let this kernel be denoted by 
$$K_{n}(y,x) = e^{\frac{it\hbar}{2mn}\Delta}e^{-\frac{-itV(y)}{\hbar n}}$$ 
Thus the integral Kernel for the Schr\"{o}dinger evolution operator is given by
$$U_{t}(y,x) = \lim_{n \to \infty} \int \; _{\spdddot} \; \int K_{n}(y,x_{n-1}) \; _{\spdddot} K_{n}(x_{2},x_{1})K_{n}(x_{1},x) \; dx_{n-1} \; _{\spdddot} \; dx_{1}$$
Applying the Fourier Transform to the free Schrodinger Evolution operator: $e^{\frac{it\hbar}{2mn}\Delta}$ allows us via the inverse Fourier Transform to write down an explicit expression for the Integral Kernel:
$$U_{t}(y,x) = \lim_{n \to \infty} \int \; _{\spdddot} \; \int e^{\frac{iS_{n}}{\hbar}}(\frac{2\pi i \hbar t}{mn})^{-nd/2} \; dx_{n-1} \; _{\spdddot} \; dx_{1}$$
Where
$$S_{n} = \sum_{k = 0}^{n-1} \frac{mn |x_{k+1} - x_{k}|^{2}}{2t} - \frac{V(x_{k+1})t}{n}$$
and $x_{0} = x$ and $x_{n} = y$.
\\

We now define the piecewise linear function $\varphi_{n}$ such that: $\varphi_{n}(0) = x$, $\varphi_{n}(t/n) = x_{1}$, $_{\spdddot}$ , $\varphi_{n}(t) = y$. We then have
$$S_{n} = \sum_{k = 0}^{n-1} [\frac{m |\varphi_{n}((k+1)t/n) - \varphi_{n}(kt/n)|^{2}}{2(t/n)^{2}} - \frac{V(\varphi_{n}((k+1)t/n)}{t/n}] (t/n)$$
It is interesting to note that $S_{n}$ is the Riemann sum of the classical action for the paths $\varphi_{n}$
$$S(\varphi, t) = \int_{0}^{t} \frac{m}{2}|\dot{\varphi(s)}|^{2} - V(\varphi(s)) \; ds$$
Hence we derive the following representation formula for the integral kernel of the Schr\"{o}dinger Evolution Operator

$$U_{t}(y,x) = \lim_{n \to \infty} \int_{P_{x,y,t}^{n}} \; e^{\frac{iS_{n}}{\hbar}}\; D\varphi_{n}$$
Where 
$$P_{x,y,t}^{n} = \{\varphi_{n} \; | \varphi_{n}(0) = x, \; \varphi_{n}(t) = y, \; \varphi_{n} \; \textnormal{is linear on the intervals} \; (kt/n, (k+1)t/n) \; \forall k = 0,1, _{\spdddot}, n-1 \}$$
$$D\varphi_{n} = (\frac{2\pi i \hbar t}{mn})^{-nd/2} \; d\varphi_{n}(t/n) \; _{\spdddot} \; d\varphi_{n}((n-1)t/n)$$
We heuristically consider the limit as $n \to \infty$. In this situation we say that $\varphi_{n}$ approaches a general path $\varphi$ from $x$ to $y$, and $S_{n} \to S(\varphi)$. We formally write
$$U_{t}(y,x) = \int_{P_{x,y,t}} \; e^{\frac{iS(\varphi,t)}{\hbar}}\; D\varphi$$
Where
$$P_{x,y,t} = \{\varphi : [0,t] \to \mathbb{R}^{n} \; | \varphi(0) = x, \; \varphi(t) = y, \; \int_{0}^{t}|\dot{\varphi(s)}|^{2} < \infty \}$$
It may be interesting to note at this point that Witten using the Feynman Path Integral and a WKB approximation (to account for the tunneling phenomena of Quantum Particles) is able to construct a chain complex which helps him to refine the Morse Inequalities. In fact one can show that the homology of this chain complex actually coincides with the singular homology of M. The idea is that the $k^{th}$ chain group is defined to to be the free abelian group generated by critical points of the Morse function f with index k. The boundary operator $\partial_{*}$ is defined to be 
$$\partial_{k+1}(q) = \sum_{p \in M_{k}(f)} \; n(q,p)p$$
where $q \in M_{k+1}(f)$  and the sum is taken over all critical points p of index k. The number $n(q,p) \in \mathbb{Z}$ is the algebraic sum of the signed gradient flow lines from q to p. These flow lines are supposed to represent the path a Quantum Mechanical particle would take traveling between the two points. He calls these paths Instantons which can be mathematically represented by 1-d integral manifolds of $\nabla f$ which connect the two critical points. This is the exact moment in the analysis where Witten uses the Feynman Path Integral. It may also be useful to note that a lot of this work was subsequently generalized into Morse Homology.  
\\
\\
Let us now build upon our Quantum Theory by considering a more recent theoretical hypothesis about our Physical world. In particular we are interested in understanding the basic theory of Supersymmetry.  Without entering into elaborate Mathematical structures such as Graded Vector Spaces, Lie Super Algebras, and Super Manifolds (though calling the Mathematics here exciting is a bit of an understatement), we hope to convince the reader of the beauty and usefulness of incorporating this structure into our mathematical universe.  
\\
\\
In the more elaborate theory of Supersymmetry, one defines $\textit{Hermitian Supercharges}$ $Q_{\mu a\alpha}$ in an "N-Extended Super Poincar\'{e} Lie Algebra" where $\mu$ is a "vector index", $a$ is a "spin index", and $\alpha$ is an "internal index". Furthermore we have the fundamental commutation relations for each of our Supercharges
\\
$$\{Q_{\mu a \alpha},Q_{\nu b \beta}\} = 2g_{\mu \nu}\Sigma_{ab}\delta_{\alpha \beta}P_{\nu}$$
\\
Here $\Sigma$ represents a bi-linear form in the spinor indices and $P_{\nu}$ is the four momentum. For our considerations we will not need the full strength of the Supersymmetric theory, instead we modestly restrict to the $0$-dimensional case. This means we consider 0 space dimensions and 1 time dimension. In this case the Supercharges satisfy more simply:
$$\{Q_{\alpha}, Q_{\beta}\} = 2\delta_{\alpha \beta}P_{0}$$
Where $P_{0}$ is just the first component of the 4-vector which represents the energy. This will be denoted by H.  Some immediate consequences are
\\
$$[H,Q_{\alpha}] = [Q_{\alpha}^{2}, Q_{\alpha}] = 0$$
$$H = Q_{\alpha}^{2}$$
\\
Hence we see that the Hamiltonian $H$ is a Positive Hermitian Operator. Furthermore we have $ker(H) = \cap_{i} ker(Q_{i})$. This last observation actually helps to formulate one of the most important problems in Supersymmetry:  
If the equation $H\psi = 0$ has a non trivial solution, a straightforward calculation shows that the state is invariant under the Supercharges.  Similarly if the equation has no non-trivial solutions then the state is not invariant under the Supercharges. If we are in the latter situation we say that our Supersymmetry is "spontaneously broken". Due to the current experimental evidence, it is highly likely that if our Physical world obeys Supersymmetry then it must be broken. This is the Physical question that motivated Witten. Having given a brief outline of Supersymmetry we now expand some effort to construct a Supersymmetry model. In particular we consider $\mathcal{N} =2$ SUSY model in $(1+0)$ Dimensions. 
\\
\\
Let $\mathcal{H}$ be a given Hilbert Space, let $Q$ be a given self-adjoint operator, and P a given bounded self-adjoint operator on $\mathcal{H}$ such that 
\\
$$Q_{1} = Q$$
$$Q_{2} = iQP$$
$$H = Q^{2}$$
$$\{Q,P\} = 0$$
$$[P, H] = 0$$
\\
We say that such a Physical system $(\mathcal{H},Q,P)$ has Supersymmetry. Since P is a self-adjoint operator and squares to 1 the only eigenvalues are $\pm1$. Hence this provides us with a natural decomposition of our Hilbert Space $\mathcal{H}$: 
\\
$$\mathcal{H}_{f} = \{\varphi \in \mathcal{H} \; | \; P\varphi = -\varphi\}$$  
$$\mathcal{H}_{b} = \{\varphi \in \mathcal{H} \; | \; P\varphi = \varphi\}$$  
$$\mathcal{H} = \mathcal{H}_{f}  \oplus \mathcal{H}_{b}$$
\\
The states in $\mathcal{H}_{f}$ are called Fermions and the states in $\mathcal{H}_{b}$ are called Bosons. This is exactly to remind us of the Fermionic and Bosonic states particles can take. The point we will try to make explicit is that SUSY is a symmetry that relates these two different particle states. 
\\
\\
Under the above decomposition of our Hilbert Space we have the following simplifications
\\
\[P =  \left( \begin{array}{ccc}
\mathbb{I}_{b} & 0 \\
0 & -\mathbb{I}_{f}  \\
\end{array} \right)\] 
\\
In particular we get the following decomposition of the algebra of operators acting on the Hilbert Space $\mathcal{H} = \mathcal{H}_{f}  \oplus \mathcal{H}_{b}$. Let $T = \left( \begin{smallmatrix}
A & B \\
C & D
\end{smallmatrix}\right)$ be an arbitrary operator. The point is that:
\\

$T$ is Bosonic or Even $\iff [P,T] = 0 \iff$  
\[T =  \left( \begin{array}{ccc}
A & 0 \\
0 & D  \\
\end{array} \right)\]  

$T$ is Fermionic or Odd $\iff \{P,T\} = 0 \iff$  
\[T =  \left( \begin{array}{ccc}
0 & B \\
C & 0  \\
\end{array} \right)\]  
\\
This identification gives us some results on our operator $Q$. Because it is Hermitian and anti-commutes with $P$ we get the following decomposition:
\\
\[Q =  \left( \begin{array}{ccc}
0 & A^{*} \\
A & 0  \\
\end{array} \right)\]  

Let $\psi = \psi_{f} \oplus \psi_{b} \in \mathcal{H}$. Then
\[
Q\psi =  \left( \begin{array}{ccc}
0 & A^{*} \\
A & 0  \\
\end{array} \right)
\left( \begin{array}{ccc}
\psi_{b} \\
\psi_{f}\\
\end{array} \right)
=
\left( \begin{array}{ccc}
A^{*}\psi_{f} \\
A\psi_{b}  \\
\end{array} \right)
\]  
\\
Since $Q: \mathcal{H} \to \mathcal{H}$ we have shown that Q interchanges Fermionic and Bosonic states. 
\\
$$Q: \mathcal{H}_{f} \to \mathcal{H}_{b}$$ 
$$Q: \mathcal{H}_{b} \to \mathcal{H}_{f}$$ 
\\
We also have a nice characterization of our Hamiltonian
\\
\[H =  \left( \begin{array}{ccc}
A^{*}A & 0 \\
0 & AA^{*}  \\
\end{array} \right)\]  
\\
We now state a fundamental result in the Mathematical theory of Supersymmetry. It says that the non-zero eigenstates have the same number of Bosonic and Fermionic states
\begin{theorem} Suppose $(H,P,Q)$ has Supersymmetry. Then for any bounded open set $\Omega \subseteq (0,\infty)$ we have
$$dim(E_{\Omega}(H)|_{\mathcal{H}_{b}}) = dim(E_{\Omega}(H)|_{\mathcal{H}_{f}})$$
Where $E_{\Omega}(H)$ is the spectral projection of H onto $\Omega$. 
\end{theorem}
Since the rigorous justification of this theorem relies on giving a precise notion of Spectral Projections we content ourselves with the following motivating remark: Let $\psi$ be an eigenstate with energy level $E > 0$. Since $[H,Q] = 0$ we have
\\
$$ H(Q\psi) = QH(\psi) = Q(E\psi) = E(Q\psi)$$
\\
This small computation tells us that for each $\psi \in \mathcal{H}_{f}$ we actually have what Physicists call a "Superpartner" $Q\psi \in \mathcal{H}_{b}$ with the same energy as $\psi$. So what one expects is that for all eigenstates with $E > 0$ they come in pairs! This is certainly not a proof, but it does help build our intuition. 
\\
\\
We now come to the primary example we will consider in this paper of a Supersymmetric Quantum Mechanical system. Actually, most of what we want to show we have already done in the previous section. We recall the system for convenience:
\\
$$\mathcal{H} = \bigoplus_{p = 0}^{n} \varOmega^{p}$$
$$ H = \Delta$$
$$Q = d + d^{*}$$
$$P_{\varOmega^{p}} = (-1)^{p}$$
\\
We already saw before that these operators satisfy the conditions for Supersymmetry. 
$$Q^{2} = \Delta$$
$$P^{2} = 1$$
$$\{Q,P\} = 0$$
\\
The operator $P$ similarly decomposes our Hilbert Space, in this case the space is "graded" into the space of even and odd forms
\\
$$\mathcal{H}_{b} = \bigoplus_{p = even} \varOmega^{p}$$
$$\mathcal{H}_{f} = \bigoplus_{p = odd} \varOmega^{p}$$ 
\\
Our operator Q can be analogously decomposed into 
\\
\[Q =  \left( \begin{array}{ccc}
0 & d^{*} \\
d & 0  \\
\end{array} \right)\]  
\\
Finally in this setting our fundamental result about Supersymmetry reads as follows: For $E \neq 0$
\\
$$\sum_{p \; odd} dim(ker[(\Delta - E)|_{\Omega^{p}}]) = \sum_{p \; even} dim(ker[(\Delta - E)|_{\Omega^{p}}]) $$
\section{Semi-Classical Approximation of Schr\"{o}dinger Operators}
As we saw in the last section, there is a rich mathematical theory surrounding Quantum Mechanical particles. In particular to study the dynamics of a particle one looks at the theory of self-adjoint operators in a Hilbert Space. In this section we consider a finer analysis of this problem. More specifically we will be interested in the semi-classical analysis of the eigenvalues of an appropriately chosen Schr\"{o}dinger Operator. We start by analyzing the operator in $\mathbb{R}^{n}$ then we comment on how to carry out this analysis on a compact manifold. We follow the arguments in \cite{CFKS} and \cite{S}.
\\
\\
We will be interested in discussing the semi-classical eigenvalue limit of self-adjoint operators of the form
$$H(\lambda) = -\Delta + \lambda^{2}h + \lambda g$$
Here we let $h,g \in C^{\infty}(\mathbb{R}^{n})$, g bounded, $h \geq 0$, $h > c > 0$ outside a compact set. Also we assume that h vanishes at only finitely many points $(x^{a})_{a=1}^{k}$ and that the Hessian of $h$ when evaluated at these points is strictly positive definite for every $k$. The objective is to estimate the eigenvalues of $H(\lambda)$ for large $\lambda$. The physical intuition here is that for large $\lambda$ the first term in the potential energy dominates and hence the operator begins to look more and more like finitely many harmonic oscillators centered at the zeros of h and separated by large barriers. Hence for large $\lambda$ and each critical point we expect $H(\lambda)$ to look like a direct sum of operators of the form
$$H^{a}(\lambda) = -\Delta +  \frac{\lambda^{2}}{2} \sum_{ij} \frac{\partial^{2}h}{\partial x_{i} \partial x_{j}} (x - x^{a})_{i}(x - x^{a})_{j} + \lambda g(x^{a})$$
We also define
$$K^{a}(\lambda) = -\Delta +  \frac{1}{2} \sum_{ij} \frac{\partial^{2}h}{\partial x_{i} \partial x_{j}} (x - x^{a})_{i}(x - x^{a})_{j} + g(x^{a})$$
Notice that up to the constant $g(x^{a})$, $K^{a}$ is the Harmonic Oscillator! By collecting the eigenvalues of $K^{a}$ $ \; \forall \; \; a = \{0,1, _{\spdddot}, k\} $ we obtain the spectrum of $$\bigoplus_{a} K^{a}$$
acting in $$\bigoplus_{a} L^{2}(\mathbb{R}^{n})$$ We also have the property:
$$\sigma(\bigoplus_{a} K^{a}) = \bigcup_{a} \sigma(K^{a})$$

Let us now state the main result concerning semi-classical limits
\begin{theorem} Let $H(\lambda)$ and $\sigma(\bigoplus_{a} K^{a})$ be as above. Let $E_{n}(\lambda)$  denote the $n^{th}$ eigenvalue of $H(\lambda)$ counting multiplicity. Denote $e_{n}$ as the $n^{th}$ eigenvalue for  $\sigma(\bigoplus_{a} K^{a})$ counting multiplicity. Then for a fixed n and $\lambda$ large, $H(\lambda)$ has at least n eigenvalues, and 
$$\lim_{\lambda \to \infty} \frac{E_{n}(\lambda)}{\lambda} = e_{n}$$
\end{theorem}  
Remark: What this theorem shows, is that to leading order, the $n^{th}$ Eigenvalue $E_{n}(\lambda)$, is completely determined by the $n^{th}$ Eigenvalue of the approximating Harmonic Oscillators. We note for completeness that there also exist results that give the complete asymptotic expansion for $E_{n}(\lambda)$ in powers of $\lambda^{-m}$ for $m \geq -1$.   
\\
\\
Before giving a proof of this statement we will have to state two results from the analysis of Schr\"{o}dinger Operators. We start with a definition. 
\begin{definition} A family of fuctions $\{J_{a}\}_{a}$ indexed by a set A is a called a Partition of Unity if 
$$ (i) \; 0\leq J_{a}(x) \leq 1 \; \; \forall x \in \mathbb{R}^{n} $$
$$ (ii) \; \sum J_{a}^{2}(x) = 1 $$
$$ (iii) \; \{J_{a}\} \; \textnormal{is locally finite} \; $$
$$ (iv) \; J_{a} \in C^{\infty} $$
$$ (v) \; \sup_{x \in \mathbb{R}^{n}} \sum_{a} |\nabla J_{a}(x)|^{2} < \infty $$
\end{definition}

\begin{theorem} (IMS Localization Formula) Let $\{J_{a}\}_{a}$ be a partition of unity and let $H = H_{0} + V$ for an appropriately decaying potential V. Then 
$$H = \sum_{a} J_{a}HJ_{a} -  \sum_{a} |\nabla J_{a}|^{2}$$
\end{theorem}
Remark: The condition on V has to be chosen so to guarantee that the state $J_{a}\psi$ remains in the domain of the Hamiltonian given $\psi$ lies in the domain.
\begin{proof}
The proof is straightforward. Compute directly the commutation relation $[J_{a}, [J_{a},H]]$, which tells us
$$[J_{a}, [J_{a},H]] = - (\nabla J_{a})^{2}$$
$$[J_{a}, [J_{a},H]] = J_{a}^{2}H + HJ_{a}^{2} - 2J_{a}HJ_{a}$$
To finish, sum over all values of $a$ to get the desired conclusion.
\end{proof}

The next result we state without proof, but we comment on its importance to the theory of Schr\"{o}dinger Operators. It will be used in a crucial way in the proof of our main theorem. 

\begin{theorem} Let V be a bounded operator. Then 
$$\inf \sigma_{ess}(H) = \sup_{K \subseteq \mathbb{R}^{n} \; compact} \; \; \inf_{\varphi \in C_{0}^{\infty}(\mathbb{R}^{n} \setminus K) \; \|\varphi \| = 1} <\varphi, H\varphi>$$
Where $ \sigma_{ess}(H) $ is the essential/continuous spectrum of $H$.   
\end{theorem}

Let us now sketch a proof of our main theorem. 
\begin{proof} 
The proof will be shown in two steps. In the first step we control the $\limsup_{\lambda \to \infty} \frac{E_{n}(\lambda)}{\lambda}$ from above by $e_{n}$. In the second step we control $\liminf_{\lambda \to \infty} \frac{E_{n}(\lambda)}{\lambda}$ from below by $e_{n}$.
\\
$\textbf{Step 1}$ Claim:  $\limsup_{\lambda \to \infty} \frac{E_{n}(\lambda)}{\lambda} \leq e_{n}$
\\
We define $J \in C_{0}^{\infty}(\mathbb{R}^{n})$ with $0 \leq J \leq 1$, 
$$ J(x) =  
\begin{cases}
1, & \text{if }|x| \leq 1 \\
0, & \text{if} \; |x| \geq 2
\end{cases}
$$
Let 
$$J_{a}(\lambda) = J(\lambda^{2/5}(x-x^{a})) \; \; \forall a = \{0,1, _{\spdddot},k\}$$
\\
$$ J_{0}(\lambda) = \sqrt{1 - \sum_{a=1}^{k}[J_{a}(\lambda)]^{2}}$$
Notice that for large enough $\lambda$, 
$$\sum_{a=0}^{k}J_{a}(\lambda)^{2} = 1$$
What we have done is construct a partition of unity $\{J_{a}\}_{a=0}^{k}$. The idea will be to understand the spectrum in localized regions of space around a critical point.
\\
Let us now fix $a \in \{0,1, _{\spdddot},k\}$. We want to apply our partition of unity to study how the approximate Hamiltonian differs from the actual Hamiltonian in a neighborhood of a critical point. We wish to look at 
$$\|J_{a}(\lambda) [H(\lambda) - H^{a}(\lambda)] J_{a}\|$$
Note that the Laplacian cancels out so we are left with two terms. Let us look at the first term:
$$|\lambda^{2}J_{a}(\lambda) [h - \frac{1}{2} \sum_{ij} \frac{\partial^{2}h}{\partial x_{i} \partial x_{j}} (x - x^{a})_{i}(x - x^{a})_{j}] J_{a}(\lambda)|$$ 
$$ = \lambda^{2}J_{a}(\lambda)^{2} |x-x^{a}|^{3} \frac{h - \frac{1}{2} \sum_{ij} \frac{\partial^{2}h}{\partial x_{i} \partial x_{j}} (x - x^{a})_{i}(x - x^{a})_{j}}{|x-x^{a}|^{3}}$$
Now using the compact support of $J_{a}(\lambda)$ and the $2^{nd}$ order Taylor Expansion of h
$$ \leq \lambda^{2} \centerdot \lambda^{-6/5} \centerdot  O(1) = O(\lambda^{4/5})$$
Similarly for the second term one can show
$$\lambda J_{a}(\lambda) [g(x) -g(x^{a})] J_{a}(\lambda) = O(\lambda^{3/5})$$
Lets us now list some properties of $H^{a}(\lambda)$ and $K^{a}(\lambda)$. Recalling the translation and dilation operators on the space $L^{2}(\mathbb{R}^{n})$ one can show that $\lambda K^{a}(\lambda)$ is unitarily equivalent to $H^{a}(\lambda)$. Furthermore since $K^{a}(\lambda)$ is just the Harmonic Oscillator up to a constant, we can completely characterize its eigenvalues and eigenfunctions
$$\sigma(K^{a}) = \{[\sum_{i =1}^{n}\omega_{i}^{a}(2n_{i} + 1)] + g(x^{a}) \; | \; n_{1} , _{\spdddot}, n_{n} \in [0,1,2,_{\spdddot}] \}$$
$$ \psi_{n}(x) = p(x)e^{-1/2 \sum_{i =1}^{n} \omega_{i}^{a} <x, \nu_{i}^{a}>^{2}}$$
Where $\{(\omega_{i}^{a})^{2}\}$ are the eigenvalues of the matrix $[A_{ij}^{a}]$,  $p(x)$ is a polynomial and $\{\nu_{i}^{a}\}$ are the orthonormal eigenvectors of $[A_{ij}^{a}]$. Hence by the unitary transformation described above, the eigenvectors of $H^{a}(\lambda)$ are
$$ p(\lambda^{1/2}(x-x^{a}))e^{-1/2 \sum_{i =1}^{n} \omega_{i}^{a} <\lambda^{1/2}(x-x^{a}), \nu_{i}^{a}>^{2}}$$
Now fix a non-negative integer n. We know that there exists $a(n)$ and $\psi_{n}$ such that 
$$H^{a(n)}(\lambda) = \lambda e_{n}\psi_{n}$$
Define
$$\varphi_{n} = J_{a(n)}(\lambda)\psi_{n}$$
The point is that the $\{\varphi_{n}\}$ up to some exponentially decaying term form an orthogonal set locally around each critical point. Fix $n \neq m$. Then, suppose $a(n) \neq a(m)$. By disjointness of the supports of the $J_{a}'s$, for $\lambda$ large enough, we know that they are orthogonal. Suppose $a(n) = a(m)$, then we know that the eigenvectors of $K^{a}$ form an orthonormal set. Hence,
$$|<\varphi_{n},\varphi_{m}> - \delta_{nm}| = \int [1-J_{a(n)}^{2}(\lambda)]\psi_{n} \psi_{m}$$
$$ \leq \int_{|x-x^{a}| \geq \lambda^{-2/5}}^{\infty}|p_{n}p_{m} e^{-1/2 \sum_{i =1}^{n} \omega_{i}^{a} <\lambda^{1/2}(x-x^{a}), \nu_{i}^{a}>^{2}}| = O(e^{-c\lambda^{1/5}})$$
Now using the localization formula, the gradient estimate $\sup_{x}|\nabla J_{a}|^{2} = O(\lambda^{4/5})$ and the estimate on the localized eigenvectors, we find that 
$$<\varphi_{n}, H(\lambda) \varphi_{m}> = \lambda e_{n}\delta_{nm} + O(\lambda^{4/5})$$
Recall the minmax formula:
$$\forall n \in \mathbb{N} \; \; \mu_{n}(\lambda) = sup_{\zeta_{1}, _{\spdddot}, \zeta_{n-1}} \ Q(\zeta_{1}, _{\spdddot}, \zeta_{n-1}; \lambda)$$
Where 
$$Q(\zeta_{1}, _{\spdddot}, \zeta_{n-1}; \lambda) = inf{\{ <\psi, H\psi> \; | \; \psi \in D(H(\lambda)), \;\; \|\varphi \| = 1\; \; \psi \in \{\zeta_{1}, _{\spdddot}, \zeta_{n-1}\}^{\bot} \}}, $$
Fixing $\epsilon > 0$, we know that for each $\lambda$ we can choose $\{\zeta_{1}^{\lambda}, _{\spdddot}, \zeta_{n-1}^{\lambda}\}$ such that 
$$\mu_{n}(\lambda) \leq Q(\zeta_{1}^{\lambda}, _{\spdddot}, \zeta_{n-1}^{\lambda}; \lambda)$$   

Now having shown that up to an exponentially decaying term the $\{\varphi_{n}\}$ form an orthogonal set it follows that for a large enough $\lambda$ they actually span an n-dimensional space. Hence we can use this new span to construct an element $\varphi \in \{\zeta_{1}^{\lambda}, _{\spdddot}, \zeta_{n-1}^{\lambda}\}^{\bot}$ for each sufficiently large $\lambda$. Using the estimate on the expectation value of the energy and the definition of Q we have:
$$Q(\zeta_{1}^{\lambda}, _{\spdddot}, \zeta_{n-1}^{\lambda}; \lambda) \leq \; <\varphi, H(\lambda) \varphi> \; \leq \lambda e_{n}\delta_{nm} + O(\lambda^{4/5})$$ 

Since $\epsilon$ is arbitrary we have
$$\mu_{n} \leq \lambda e_{n}\delta_{nm} + O(\lambda^{4/5})$$ 

Using the Poincar\'{e} Inequality, the assumptions on the functions h and g, and the max-min characterization of the spectrum we also have
$$\inf \sigma_{ess}(H) \geq c\lambda^{2} \; \; \; \textnormal{where} \;  c>0$$
There is a basic fact about that max-min value that says either $H(\lambda)$ has n eigenvalues and $\mu_{n}(\lambda) = E_{n}(\lambda)$ or $\mu_{n}(\lambda) = \inf \sigma_{ess}(H)$. Since we have a linear growth bound on $\mu_{n}$ from above and $\inf \sigma_{ess}(H)$ grows at least quadratically, this implies that for $\lambda$ sufficiently large $\mu_{n}(\lambda) = E_{n}(\lambda)$. Hence we have
$$E_{n}(\lambda) \leq \lambda e_{n}\delta_{nm} + O(\lambda^{4/5})$$ 
This proves Step 1. 
\\
\\
$\textbf{Step 2}$ Claim:  $\liminf_{\lambda \to \infty} \frac{E_{n}(\lambda)}{\lambda} \geq e_{n}$
\\
\\
To show the lower bound we actually claim that it suffices to show for any $e$ not in $\sigma(\bigoplus_{a} K^{a})$, for example $e \in (e_{m}, e_{m+1})$ 
$$H(\lambda) \geq \lambda e + R + o(\lambda)$$
where R is a rank m symmetric operator. To see why this implies the estimate, suppose 
$$H(\lambda) \geq \lambda e + R + o(\lambda) \; \; \forall m$$
with $e_{m+1} > e_{m}$ and $e \in (e_{m}, e_{m+1})$. Simply pick a vector $\psi$ in the span of the first m+1 eigenfunctions of $H(\lambda)$ with $\|\psi\| = 1$ and $\psi \in ker(R)$. We have
$$E_{m+1} \geq \; <\psi, H\psi> \; \geq \lambda e + o(\lambda)$$
And we have our desired estimate for $n = m+1$. Similarly one can argue the same if $e_{m+1} = e_{m}$.
\\
\\
Now fix $e \in (e_{m}, e_{m+1})$. We apply the IMS localization theorem to get
$$H(\lambda) = \sum_{a} J_{a}HJ_{a} -  \sum_{a} |\nabla J_{a}|^{2}$$
$$ = J_{0}HJ_{0} + \sum_{a} J_{a}H^{a}J_{a} + O(\lambda^{4/5})$$
Because of the quadratic convergence of h to its zeros, and the fact that $J_{0}$ has support away from the ball of radius $\lambda^{-2/5}$ around each critical point of h, we see that on the support of $J_{0}$
$$h(x) \geq c(\lambda^{-2/5})^{2} = c\lambda^{-4/5}$$
Hence $J_{0}HJ_{0}$ grows like $\lambda^{2} \centerdot \lambda^{-4/5} = \lambda^{6/5}$ for large $\lambda$. Thus for an even larger $\lambda$ since the exponential power is strictly greater than 1
$$ J_{0}HJ_{0} \geq \lambda e J_{0}^{2} $$
For $a \neq 0$ one has
$$ J_{a}H^{a}J_{a}  \geq J_{a}R^{a}J_{a} +  \lambda e J_{a}^{2}$$ 
Where $R^{a}(\lambda)$ is the restriction of $H^{a}$ to the span of all eigenvectors of $H^{a}$ with eigenvalues below $\lambda e$. We also have $Rank(J_{a}R^{a}J_{a}) \leq Rank(R^{a})$ where $R^{a}$ is the number of eigenvalues of $K^{a}$ below $e$. Now we recall a fact from Linear Algebra $Rank(A+B) \leq Rank(A) + Rank(B)$. So
$$Rank[\sum_{a} J_{a}R^{a}J_{a} ] \leq \sum_{a} Rank(J_{a}R^{a}J_{a})$$ 
$$= \{\textnormal{Number of Eigenvalues of} \; \bigoplus_{a}K^{a} \; \textnormal{below e} \} = n$$
Hence
$$ H(\lambda) \geq \lambda e J_{0}^{2}  + \lambda e  \sum_{a}  J_{a}^{2} +  \sum_{a}  J_{a}R^{a}J_{a} + O(\lambda^{4/5}) \geq \lambda e + R + O(\lambda^{4/5})$$ 
\end{proof}

We conclude our discussion of the semi-classical analysis with some heuristic remarks how one goes about extending the above results to compact finite dimensional manifolds. As we will see in the next section Witten's analysis requires an application of the semi-classical limit to a Schr\"{o}dinger type operator acting in a Hilbert Space where the underlying topological space is a compact finite dimensional manifold. 
\\
\\
There are three major hurdles one needs to overcome when applying the above results to operators of the form $L + \lambda^{2}h + \lambda g$ acting on $\Omega^{p}$ where we have taken L to be the Laplace Beltrami Operator. 
\\
1. First we need to extend the IMS Localization Formula
\\
2. $L \neq \Delta$ in local coordinates where $-\Delta$ is the Euclidean Laplacian
\\
3. h and g may have nontrivial vectoral dependence. 
\\
\\
1. Let $a^{*}(f)$ be the wedge product with $df$ and $a(f)$ its adjoint. Then $\{a^{*},a\} = |df|^{2}$. One sees that for $f$ a given function and $\omega$ a given p-form
$$[d,f\omega] =  a^{*}(f)$$
$$[d^{*},f\omega] = -a(f)$$
$$[f\omega, [f\omega,L]] = -2|df|^{2}$$
Hence one gets an IMS localization formula of the form
$$L = \sum J_{\alpha}LJ_{\alpha} - \sum |dJ_{\alpha}|^{2}$$
\\
2. As we will see in the next section, locally around every contact point we can take the metric to be flat. Hence locally $L = \Delta$. 
\\
\\
3. We will also see below that after a diagonalization of $g(x^{a})$, this operator becomes a sum of scalar valued operators. Also below h is restricted to act as a multiple of the Identity matrix. 

\section{Witten's Proof of the Morse Inequalities}
In previous sections we have considered the SUSY model 
$$(H = \Delta, Q = d+d^{*}, (-1)^{p})$$ 
acting on the Hilbert Space $\mathcal{H} = \bigoplus_{p = 0}^{n} \Omega^{p}$.  To obtain the Morse Inequalities in both their strong and weak forms we will consider a deformed SUSY model by defining the $t$ dependent $(t \in \mathbb{R})$ model $(H = \Delta_{t}, d_{t} + d_{t}^{*}, (-1)^{p})$ acting on the same Hilbert Space. We define for a Morse function $f$
$$d_{t} = e^{-tf}de^{tf}$$
$$d_{t}^{*} = e^{-tf}d^{*}e^{tf}$$  
$$\Delta_{t} = d_{t}d_{t}^{*} + d_{t}^{*}d_{t}$$
One can check explicitly that all of the conditions for Supesymmetry are satisfied. Since the operators $d$ and $d^{*}$  are equal to the operators $d_{t}$ and $d_{t}^{*}$ respectively up to conjugation we see immediately that
$$H_{t}^{p} = ker(d_{t}: \varOmega^{p} \to \varOmega^{p+1}) \setminus im(d_{t}: \varOmega^{p-1} \to \varOmega^{p})$$
$$ = e^{-tf}ker(d_{t}: \varOmega^{p} \to \varOmega^{p+1}) \setminus e^{-tf}im(d_{t}: \varOmega^{p-1} \to \varOmega^{p})$$
$$ = ker(d: \varOmega^{p} \to \varOmega^{p+1}) \setminus im(d: \varOmega^{p-1} \to \varOmega^{p}) = H^{p} $$
On the other hand since we defined the deformed Laplacian we can extend the results of Hodge Theory and conclude that
$$Ker[\Delta_{t}|_{\Omega^{p}}] \simeq H_{t}^{p}$$
Hence what we have shown is the independence of the $t$ parameter in the dimension of the deformed Cohomological space. More specifically we have
$$\beta_{p} = dim(ker[\Delta_{t}|_{\Omega^{p}}])$$
What this result tells us is that to estimate $\beta_{p}$ it suffices to estimate the dimension of  $Ker[\Delta_{t}|_{\Omega^{p}}]$ for any $t$. The fundamental observation made by Witten was that the $t \to \infty$ limit is the semi-classical limit for the operator $\Delta_{t}$. More concretely the asymptotic estimates for the eigenvalues of $\Delta_{t}$ can be used to estimate the dimension of $Ker[\Delta_{t}|_{\Omega^{p}}]$ for large $t$.  

In the course of our analysis it will be useful to define a few new operators, whose action locally on p-forms is defined as follows
$$(a^{i})^{*}\alpha = dx^{i} \wedge \alpha$$
The adjoint action $a_{i}$ can be directly calculated
$$a^{i}dx^{j_{i}} \wedge _{\spdddot} \wedge dx^{j_{p}} = \sum_{k=1}^{p} (-1)^{k}g^{ij_{k}}dx^{j_{i}} \wedge _{\spdddot} \wedge dx^{j_{k-1}} \wedge dx^{i_{k+1}} \wedge _{\spdddot}dx^{j_{p}} $$ 
Using these two facts one can easily convince themselves that
$$\{a^{i}, (a^{j})^{*}\} = g^{ij}$$
These newly defined objects should remind us of the Fermion creation and annihilation operators that arise in Physics. We now provide an expression for $\Delta_{t}$ which highlights the critical points of our Morse function $f$.  We now proceed to prove the Morse Inequalities We follow the arguments in \cite{CFKS}. 
\begin{proposition} $$\Delta_{t} = \Delta + t^{2}\|df\|^{2} + tA$$ Where A is the $zero^{th}$ order operator, i.e. A acts as multiplication by smooth functions, represented by
$$A = \sum_{ij}(\frac{\partial^{2}f}{\partial x^{i} \partial x^{j}}) [(a^{i})^{*}, a^{j}]$$ in a flat neighborhood with respect to an orthonormal coordinate system.
\end{proposition}
\begin{proof}
The proof is a direct calculation. 
$$d_{t} \alpha = d \alpha + t \; df \wedge \alpha$$ Where $$df = \sum_{i} \frac{\partial f}{\partial x^{i}} dx^{i}$$ We see that
$$d_{t} = d + t\sum_{i}^{n} \frac{\partial f}{\partial x^{i}} (a^{i})^{*}$$ Hence
$$d_{t}^{*} = d^{*} + t\sum_{i}^{n} \frac{\partial f}{\partial x^{i}} (a^{i})$$
This implies
$$\Delta_{t} = \{d_{t},d_{t}^{*}\} = \Delta + t^{2} \sum_{ij}\frac{\partial f}{\partial x^{i}} \frac{\partial f}{\partial x^{j}} \{a^{i}, (a^{j})^{*}\} + tA$$
Where 
$$A = \{d, \sum_{j}\frac{\partial f}{\partial x^{j}} a^{j} \} + \{d, \sum_{j}\frac{\partial f}{\partial x^{j}} a^{j} \}^{*}$$
Our task is to show that $A$ is a $zero^{th}$ order operator. Let us define $\partial_{i}$ by
$$\partial_{i} (\sum u_{j_{i}} \; _{\spdddot}{j_{p}} dx^{j_{i}} \wedge \; _{\spdddot} \wedge dx^{j_{p}} ) = \sum \frac{\partial u_{j_{i}  \;_{\spdddot} j_{p}}}{\partial x^{i}}  dx^{j_{i}} \wedge \; _{\spdddot} \wedge dx^{j_{p}}$$
So we have
$$d = \sum_{i} (a^{i})^{*} \partial_{i} $$ 
Applying this representation to $A$ and considering a local coordinate system in a flat neighborhood, we observe:
$$[\partial_{i}, a^{j}] = 0$$
$$A = \sum_{ij} (\frac{\partial^{2}f}{\partial x^{i} \partial x^{j}})  [2(a^{i})^{*}a^{j} - \delta^{ij}]$$ This gives us the conclusion of the proposition after noting that $\{(a^{i})^{*}, a^{j}\} = \delta^{ij}$ when the metric is flat.  
\end{proof}
The time has come to prove the Morse Inequalities using ideas from Supersymmetric Quantum Mechanics. We first show the Weak Morse Inequalities using the formalism we have developed thus far. The main idea is to apply the Manifold version of the semi-classical analysis to the following operator acting on p-forms.  
$$\Delta_{t} = \Delta + t^{2}\|df\|^{2} + tA$$
We are interested in estimating the dimension of the kernel for large values of t. This is in fact the idea inherent in Witten's approach to proving the weak Morse Inequalities. 
\begin{proof}
The first task is to choose a Riemannian metric so that in a neighborhood of a critical point we have the Morse Coordinates, and by the isolation of the critical points and hence their finiteness, we can patch these neighborhoods together with an arbitrary metric in the other regions using a partition of unity. Applying the Morse coordinates in a neighborhood of a critical point $x^{a}$ we have
$$\Delta_{t} = \Delta + 4t^{2}\sum_{i}^{n}x_{i}^{2} + 2t \sum_{i=1}^{n-\mu(x^{a})}[(a^{i})^{*},a^{i}] -2t \sum_{i= n-\mu(x^{a}) + 1}^{n}[(a^{i})^{*},a^{i}]$$   
Where $\mu(x^{a})$ represents the Morse Index of the critical point $x^{a}$. 
\\
\\ 
We also define 
$$K^{a} = \Delta + 4\sum_{i}^{n}x_{i}^{2}  + A^{a}$$
Where 
$$A^{a} = 2 \sum_{i=1}^{n-\mu(x^{a})}[(a^{i})^{*},a^{i}] -2 \sum_{i= n-\mu(x^{a}) + 1}^{n}[(a^{i})^{*},a^{i}]$$
Let us first compute the spectrum of $\bigoplus_{a}K^{a}$. The first observation to make is that $\Delta + 4\sum_{i}^{n}x_{i}^{2}$ acts as a scalar operator on p-forms. Actually what we have is the Harmonic Oscillator with eigenvalues 
$$\{\sum_{i}^{n} 2(1 +2 n_{i}) \; | \; n_{i}, _{\spdddot}, n_{n} \in \{0,1,2, _{\spdddot}\} \}$$
We recall that each eigenvalue has $\frac{n!}{(n-p)!p!}$ independent eigenvectors which can be represented by 
$$\psi \; dx_{i_{1}} \wedge _{\spdddot} \wedge dx_{i_{p}} \; \; 1 \leq i_{1} < _{\spdddot} < i_{p} \leq n$$
One should also see that
$$ [(a^{i})^{*},a^{i}](f \; dx_{i_{1}} \wedge _{\spdddot} \wedge dx_{i_{p}}) =  
\begin{cases}
f \; dx_{i_{1}} \wedge _{\spdddot} \wedge dx_{i_{p}}, & \text{if }i \in \{i_{1}, _{\spdddot}, i_{p}\} \\
-f \; dx_{i_{1}} \wedge _{\spdddot} \wedge dx_{i_{p}}, & \text{if} \; i \notin \{i_{1}, _{\spdddot}, i_{p}\}  
\end{cases}
$$
Hence if we consider its action on the eigenfunctions we see that $A^{a}$ acts diagonally:
$$A^{a} \psi \; dx_{i_{1}} \wedge _{\spdddot} \wedge dx_{i_{p}} = \gamma_{a}\psi dx_{i_{1}} \wedge _{\spdddot} \wedge dx_{i_{p}}$$
Where
$$ \gamma_{a} = |I \cap K| - |J \cap K| - |I \cap L| + |J \cap L|$$
Defining: $I = \{i_{1}, _{\spdddot}, i_{p}\}$,$\;J  = \{1, _{\spdddot}, n \} \setminus I$, $K = \{1, _{\spdddot}, n-\mu_{x^{a}}\}$, $L = \{n-\mu_{x^{a}} + 1, _{\spdddot}, n\}$. We collect these facts and note:
$$\sigma(K^{a}) = \{ \sum_{i=1}^{n} 2(1+2n_{i}) + 2\gamma_{a} \; | \; n_{1}, _{\spdddot}, n_{n} \in \{0,1,2, _{\spdddot} \} \; \; \textnormal{and} \; \; 1 \leq i_{1} < _{\spdddot} < i_{p} \leq n \}$$
Our next task is to understand the term $\gamma_{a}$.   Upon inspection what we find is that $\gamma_{a} \geq -n$. Since this gives us a trivial lower bound on the spectrum, we note that if we want to understand the multiplicity of the zero eigenvalue we must set all of $(n_{i})$ to zero (or we will instead be studying the  non-zero eigenvalues of $K^{a}$). Upon deeper considerations we find that $\gamma_{a} = -n$ precisely when $(i_{1}, _{\spdddot}, i_{p}) = (n-p+ 1, _{\spdddot}, n)$. What we conclude from these observations is $Ker(K^{a}|_{\Omega^{p}}) = 0$ unless $\mu(x^{a}) = p$. In that case $dim([Ker(K^{a}|_{\Omega^{p}}]) = 1$. This implies that 
$$dim([Ker(\bigoplus_{a}K^{a}|_{\Omega^{p}}]) = M_{p}(f)$$
Letting $E_{n}^{p}(t)$ be the eigenvalues of $\Delta_{t}|_{\Omega^{p}}$ counting multiplicity and $e_{n}^{p}$ be the eigenvalues of $\bigoplus_{a}K^{a}|_{\Omega^{p}}$ counting multiplicity. From the semi-classical analysis we know
$$\lim_{t \to \infty} \frac{E_{n}^{p}(t)}{t} = e_{n}^{p}$$
Thus
$$\beta_{p} =  dim(Ker[\Delta_{t}|_{\Omega^{p}}]) \leq dim[Ker(\bigoplus_{a}K^{a}|_{\Omega^{p}})] = M_{p}(f)$$
\end{proof}
Thus we have shown the Weak Morse Inequality. We will now apply ideas from Supersymmetry to prove the strong Morse Inequalities.
\begin{proof}
 We know that $e_{M_{p}+1}^{p}$ represents the first nonzero Eigenvalue for  $\bigoplus_{a}K^{a}|_{\Omega^{p}}$. This implies for large $t$ that $E_{n}^{p}(t)$ grows like $t$ for $ n\geq M_{p} + 1$. When considering the rest of the eigenvalues we know that the first $\beta_{p}$ are zero. We call the eigenvalues $\{E_{\beta_{p+1}}^{p}(t), _{\spdddot}, E_{M_{p}}^{p}(t) \}$ the low-lying eigenvalues. They are nonzero but $o(t)$ as $t \to \infty$. We now recall that the fundamental property of Supersymmetry tells us that the low-lying eigenvalues occur in even-odd pairs. For each low-lying eigenvalue $E_{n}^{p}$ with p even there is a low-lying eigenvalue  $E_{n}^{p'}$ with p' odd. Hence
$$\sum_{p \; odd} (M_{p} - \beta_{p}) =  \sum_{p \; even} (M_{p} - \beta_{p})$$
Which implies the Morse Index Theorem:
$$\sum_{p =0}^{n} (-1)^{p}M_{p} =   \sum_{p =0}^{n} (-1)^{p}\beta_{p} + \sum_{p \; even} (M_{p} - \beta_{p}) - \sum_{p \; odd} (M_{p} - \beta_{p})  = \sum_{p =0}^{n} (-1)^{p}\beta_{p}$$
To get the Strong Morse Inequalities we need to look closer at the Supersymmetric cancellations. Let $\Lambda_{t}^{p}$ denote the $M_{p} - \beta_{p}$ dimensional subspace of low-lying eigenvalues. Now we know that $Q_{t}^{2} = \Delta_{t}$, thus $Q_{t}$ preserves the eigenspaces of $\Delta_{t}$. Hence $Q_{t}$ is a one-to-one map and 
$$ Q_{t}: \bigoplus_{l \; odd \; l =1}^{2j-1}   \Lambda_{t}^{l} \to \bigoplus_{l \; even \; l=0}^{2j}  \Lambda_{t}^{l}$$
$$ Q_{t}: \bigoplus_{l \; even \; l= 0}^{2j}  \Lambda_{t}^{l} \to \bigoplus_{l \; odd \; l = 1}^{2j+1}  \Lambda_{t}^{l}$$ 
\\
By injectivity the dimensions of the domain must be less than or equal to the dimensions of the target space. Hence for $0 \leq 2j < n$ and $0 \leq 2j+1 < n$

$$(M_{1} - \beta_{1}) + _{\spdddot} + (M_{2j-1} - \beta_{2j-1}) \leq (M_{0} - \beta_{0})  + _{\spdddot} + (M_{2j} - \beta_{2j}) $$
$$(M_{0} - \beta_{0}) + _{\spdddot} + (M_{2j} - \beta_{2j}) \leq (M_{1} - \beta_{1})  + _{\spdddot} + (M_{2j+1} - \beta_{2j+1}) $$

\end{proof}

We conclude this section with some remarks on Witten's arguments for deriving the Strong Morse Inequalities. Using only algebraic techniques one can prove that the counting series 
$$C_{t} = \sum_{q} t^{q} \; dim(C_{q})$$ for any finite dimensional cochain complex satisfies the Polynomial Morse Inequalities relative to the Poincar\'{e} series of its Cohomology 
$$P_{t} = \sum_{q} t^{q} \; dim(H^{q}(C))$$ We essentially considered this viewpoint when we first encountered the Morse Inequalities. The point is that the Morse Inequalities do not give a canonical form for the coboundary operator, hence one is at liberty to construct such an operator and show that the Betti numbers of this Cohomology equals those of the underlying manifold M. In essence Witten's idea is to consider the vector space $X_{p}$, $p = 0, _{\spdddot}, n$ spanned by the low-lying eigenvalues considered above. One then restricts the "co-boundary" operator to this vector space, which is well defined since $d_{t}$ commutes with $\Delta_{t}$, so $d_{t}$ preserves the eigenspaces of $\Delta_{t}$. The point is that the $p^{th}$ Cohomology group of this complex has dimension $\beta_{p}$ and the dimension of $X_{p}$ is $M_{p}$. Hence one gets the Strong Morse Inequalities from our initial remarks modulo some technical points which show that the dimension of the chain complex is independent of the t parameter and the explicit construction of this boundary operator. In fact Witten's idea is to use Instantons to carry critical points of index k to critical points of index k+1 (very similar analysis to what we mentioned previously). The rigorous justification for these points was completed in the work of Helffer and Sj\"{o}strand. For a wonderful exposition about Morse Homology and Witten's construction we highly recommend Bott's paper. He even gives you a Physical example! Finally we just mention for completeness that in the literature this cochain construction goes by the name of "The Witten Complex" and is an alternate way of understanding Morse Homology.

\end{document}